\documentclass[reqno]{amsart}
\usepackage{mathbbol}
\usepackage{wrapfig}
\usepackage{epsfig}
\usepackage{epsfig,latexsym,amsfonts,amssymb,amsmath,amscd,graphics,epic}
\usepackage{amsfonts,amssymb,amsmath,amscd,amsthm}
\usepackage[mathscr]{eucal}
\usepackage{mathrsfs}
\usepackage{mathbbol}
\usepackage{oldgerm,units}
\usepackage[active]{srcltx}
\usepackage{mathabx}
\usepackage{pstricks,pst-node,pst-tree}

  \input amssymb.sty

\def\pSkip{\vskip 1mm \noindent}

\def\Id{\Pi}
\def\htId{\widehat\Id}
\def\tI{\mathcal I}

\def\htv{\widehat v}
\def\htu{\widehat u}

\def\Pmax{{\max \{  p \ds | p \in P \} }}
\def\N{\mathbb N}
\def\Z{\mathbb Z}

\def\chw{\widecheck{w}}
\def\chx{\widecheck{x}}
\def\chy{\widecheck{y}}

\def\chC{\widecheck{C}}

\def\ver{\mathcal V}
\def\arc{\mathcal E}
\def\varX{\mathcal X}

\newcommand\factor[2]{#1 | #2}

\def\pre{\operatorname{pre}}
\def\suf{\operatorname{suf}}
\def\no{\kappa}
\def\wlen{\ell}
\def\cont{\operatorname{cont}}
\def\exp{\operatorname{exp}}

\def\wset{\mathcal W}
\def\tlw{\widetilde w}
\def\tlm{\widetilde m}
\newcommand\thmref[2]{\pSkip\textbf{Theorem #1. }\emph{#2}\vskip 1mm}

\def\X{{\underline{X}}}
\def\Y{{\underline{Y}}}
\def\x{{\underline{x}}}
\def\y{{\underline{y}}}

\newcommand\word[1]{W(#1)}

\newcommand\pres[2]{#1 \setminus _{\suf} #2}
\newcommand\sres[2]{#1 \setminus _{\pre} #2}

\newcommand\cset[1]{\langle #1 \rangle}

\def\set2{{\cset{2}}}

\def\setm{{\cset{m}}}

\def\setZ{{\cset{Z}}}

\def\setL{{\cset{L}}}
\def\setR{{\cset{R}}}
\def\setLp{{\cset{L'}}}
\def\setRp{{\cset{R'}}}

\def\setB{{\cset{B}}}
\def\setBp{{\cset{B'}}}
\def\setBpp{{\cset{B''}}}

\def\setF{{\cset{F}}}
\def\setG{{\cset{G}}}
\def\setH{{\cset{H}}}
\def\setU{{\cset{U}}}
\def\setV{{\cset{V}}}

\def\fib{\operatorname{F}}

\def\l2{x^2y^2x}
\def\ll2{yx^2y^2x}
\def\l3{xyxy^2x^2y}

\newcommand{\len}[1]{\operatorname{\ell}(#1)}

\newcommand{\ds}[1]{\ {#1} \ }
\newcommand{\dss}[1]{\quad {#1} \quad }

\def\pth{\gm}
\def\grph{G}

\def\tmu{\widetilde{\mu}}
\def\tgm{\widetilde{\gm}}
\def\hgm{\widehat{\gm}}

\def\one{\mathbb{1}}
\def\zero{\mathbb{0}}
\def\diag{\operatorname{diag}}
\def\dsim{\sim_{\diag}}

\def\mT{M_n(\Trop)}
\def\uT{U_n(\Trop)}
\def\lT{L_n(\Trop)}

\newtheorem{theorem}{Theorem}[section]

\newtheorem{lemma}[theorem]{Lemma}
\newtheorem{conjecture}[theorem]{Conjecture}
\newtheorem{notation}[theorem]{Notation}
\newtheorem{corollary}[theorem]{Corollary}

\newtheorem{example}[theorem]{Example}
\newtheorem{remark}[theorem]{Remark}

\newcommand {\junk}[1]{}
\newcommand{\etype}[1]{\renewcommand{\labelenumi}{(#1{enumi})}}
\def\eroman{\etype{\roman}}

\newcommand{\bfem}[1]{\textbf{\emph{#1}}}
\def\({\left(}
\def\){\right)}

\def\al{\alpha}
\def\gm{\gamma}

\def\minf{-\infty}

\def\Real{\mathbb R}
\def\Trop{\mathbb T}

\def\grph{G}

\def\minf{-\infty}

\def\MnT{M_n(\Trop)}
\def\UnT{U_n(\Trop)}
\def\LnT{L_n(\Trop)}

\title[Semigroup identities in the monoid of triangular tropical  matrices]
 {Semigroup identities in the monoid of \\[5pt]  triangular tropical  matrices}

\author{Zur Izhakian}

\address{
 Department of Mathematics, University of Bremen,
28359 Bremen, Germany.
        }
    \email{zzur@post.tau.ac.il;zzur@math.biu.ac.il}

 \thanks{\textbf{Acknowledgement.}
The author thanks Glenn Merlet for the useful conversation in CIRM, Luminy, September 2011.
 }


\date{\today}


\begin{document}

\begin{abstract} We show that
the submonoid of all $n \times n$ triangular tropical  matrices
satisfies a nontrivial  semigroup identity and provide a generic
 construction for classes of such identities. The utilization of the Fibonacci number formula gives us an upper bound on the length of these $2$-variable semigroup identities.
\end{abstract}

\maketitle




\section{Introduction}
\numberwithin{equation}{section}

 Varieties in the classical theory are customarily determined as the solutions of
 systems of equations. The ``weak" nature of semigroups, i.e., lack of inverses, forces
the utilization of a  different approach than the familiar one, in
which semigroup identities simulate the role of equations in the
classical theory. These semigroup identities are at the heart of the theory
of semigroup varieties \cite{SV}, and have been intensively
studied for many years. A new approach for studying these  semigroup identities has been provided by the use of
tropical algebra, as introduced in \cite{IzhakianMargolisIdentity}.

Tropical algebra is carried out over the tropical semiring
$\Trop\ := \Real\cup\{-\infty\}$ with the operations of maximum
and summation (written in the standard algebraic way),
$$a +  b :=\max\{a,b\},\qquad a b :=a \underset{\operatorname{sum}}{+}b,$$
serving respectively as addition and multiplication
\cite{IMS,pin98,RST}. This semiring is an additively idempotent
semiring, i.e., $a + a = a$ for every $a\in \Trop$, in which
$\zero := -\infty$ is the zero element and $\one := 0$ is the
multiplicative  unit.

As shown in \cite{IzhakianMargolisIdentity} for the case of $2
\times 2$ tropical matrices, linear representations of semigroups
by matrices over idempotent semirings, and in particular over the
tropical semiring, establish  a useful machinery for identifying
and proving semigroup identities.
It has been also shown in  \cite{IzhakianMargolisIdentity} that
the monoid of $2 \times 2$ tropical matrices admits a nontrivial
semigroup identity. Therefore, aiming for a generalization of the
former results, it is natural to inquire about possible
nontrivial identities in the wider monoid of $n \times n$ tropical
matrices. These matrices essentially serve as the target for representing
semigroups, and thus enable representations and study of a larger
range of monoids and semigroups.

 In the present paper we deal with $n \times n$
triangular tropical matrices, improving the machinery introduced
in \cite{IzhakianMargolisIdentity} by bringing in the perspective
of graph theory. It is well known that graph theory, especially
the theory of direct graphs, is strongly related to tropical
matrices~\cite{ABG,Butk} and provides a powerful computational tool in
tropical matrix algebra;
 the correspondence between tropical matrices and weighted digraphs is intensively used for proving the results of this paper. A background on the interplay between digraphs and tropical matrices is given in \S\ref{sec:mathGr}.

Before approaching tropical matrices, we first address semigroup identities in general (cf. \S\ref{sec:3}),
providing a generic construction of classes of identities that preserve certain required properties (in particular balancing), to be used in our further study.  The following refinement property is very useful for this study, especially for proving existence of identities:
\thmref{\ref{thm:refine}}{A semigroup  that satisfies an
 $n$-variable identity,  also satisfies a refined
$2$-variable identity with exponent set $\{1,2\}$.}
\noindent This refinement of semigroup  identities assists us to deal with tropical matrices by better utilizing their view  as digraphs.

The monoid $\mT$  of $n \times n$  matrices over the tropical
semiring plays, as one would expect, an important role both in
theoretical algebraic study and in applications to combinatorics, as well as in
semigroup representations and automata. In contrast to the case
of matrices over a field,
 we identify nontrivial semigroup identities, satisfied by the
submonoid $\uT$ (resp. $\lT$) of all
 upper (resp. lower)  $n \times n$  triangular tropical matrices.

 To simplify the exposition we open with a certain type of matrices, i.e., diagonally equivalent matrices,
 and have the following preliminary theorem:

\thmref{\ref{thm:simDia}}{Any two triangular tropical matrices $X,Y \in \uT$ having
the same diagonal satisfy the (nontrivial) identities:
\begin{equation}\label{eq:00} \tlw_{(C,P,n-1)} \ds X \tlw_{(C,P,n-1)} = \tlw_{(C,P,n-1)} \ds Y
\tlw_{(C,P,n-1)},
\end{equation}  where $\tlw_{(C,P,n-1)}$ is any word having as its factors all the words
 of length $n-1$ generated by $C = \{ X, Y\}$ of powers $P = \{ 1,2\}$, such that $\tlw_{(C,P,n-1)} \ds X \tlw_{(C,P,n-1)}$ and $\tlw_{(C,P,n-1)} \ds Y
\tlw_{(C,P,n-1)}$ are generated by $C$ and powers $P$.   (To be explained in the text below.) }

Using this result, basically proved by combinatorial arguments on
the associated (colored) weighted digraphs of products of tropical matrices, we obtain
the main result of the paper.
\thmref{\ref{thm:matU}}{The submonoid $\uT \subset \mT$ of upper
triangular tropical   matrices satisfies the nontrivial semigroup identities
\eqref{eq:00}, with $X = AB$ and $Y = BA$,  for any $A, B\in \uT$.
\ }
This theorem generalizes the identity of the submonoid $U_2(\Trop)$ of $2
\times 2$ triangular tropical matrices, introduced in
\cite[Theorem~3.6]{IzhakianMargolisIdentity}, which has been
utilized to serve as the target of a faithful (linear)
representation for the bicyclic monoid; the latter plays an important role in the study of semigroups. The use of this faithful representation, together with the identity admitted by $U_2(\Trop)$,  led in \cite{IzhakianMargolisIdentity} to an easy proof
of Adjan's identity of the bicyclic monoid  (see \cite{Adjan} for Adjan's original work).

The well known Fibonacci number formula provides us an easy way to compute an upper bound for the length of the $2$-variable semigroup identities discussed in this paper (cf. \S\ref{sec:5}). This upper bound can be improved further, depending on the structure of the identities.

In the past years, most of the theory of  semigroups and matrix
representations has been developed for matrices built over fields
or rings; the above results nicely demonstrate our new approach
to represent semigroups, much along the line of group
representations, attained by a ``direct" use of matrices, realized
as linear operators, but now taking place over semirings.

 The results of this
paper opens up the possibility of using representation theory over
the tropical semiring to study wider classes of semigroups and monoids and to prove their possible (minimal) semigroup identities.

\section{Background: Tropical matrices and weighted digraphs}\label{sec:mathGr}

 Recalling  that $\Trop$
is a semiring, then in the usual way, we have the semiring $\MnT$
of $n \times n$ matrices with entries in $\Trop$, whose addition
and multiplication are induced from $\Trop$ as in the familiar
matrix construction. The \bfem{unit} element $I$ of
 $\MnT$, is the matrix
with $\one = 0$ on the main diagonal and whose off-diagonal
entries are all $\zero = \minf$; the \bfem{zero} matrix is
$(\zero) =  \zero I$. Therefore, $\MnT$ is also a multiplicative
monoid, and in the sequel it is always referred to as a monoid.
Formally,  for any nonzero matrix $A \in \MnT$ we set $A^{0} :=
I$. A given matrix $A \in  \MnT$ with entries $a_{i,j}$ is written
as $A = (a_{i,j})$, $i,j = 1,\dots,n$. We denote by $\UnT$ (resp.
$\LnT$) the submonoid of $\MnT$ of all upper (resp. lower)
triangular tropical matrices.

Given two matrices $X = (x_{i,j})$ and  $Y = (y_{i,j})$ in $\MnT$ , we
write  \begin{equation}\label{eq:dig} X  \dsim Y
\dss{\Leftrightarrow}  x_{i,i} = y_{i,i}, \quad \text{for all } i
= 1,\dots , n,
\end{equation} and say that $X$ and $Y$  are
\textbf{diagonally equivalent} if \eqref{eq:dig} holds.
\begin{remark} It is
readily checked that $AB \dsim BA$ for any upper (or lower)
triangular matrices $A$ and $B$.
\end{remark}

The \textbf{associated weighted digraph} $\grph_A := (\ver, \arc)$
of an $n\times n$ tropical matrix $A = (a_{i,j})$ is defined to
have vertex set $\ver :=\{ 1, \dots, n\}$, and edge set $\arc$ having a directed edge
$(i,j) \in \arc$ from $i$ to $j$ (of \textbf{weight} $a_{i,j}$) whenever
$a_{i,j} \ne \zero$. A \textbf{path} $\gm$ is a sequence of edges
$(i_1, j_1), \dots, (i_m, j_m) $, with $j_{k} = i_{k+1}$ for every
$k = 1,\dots, m-1$. We write $\gm := \gm_{i,j}$ to indicate that
$\gm$ is a path from $i = i_1$ to $j=j_m$, and call $\gm_{i,s}$
(resp. $\gm_{s,j}$), where  $s= i_k$ and $1 < k < m,$ the
\textbf{prefix} (resp. \textbf{suffix}) of $ \gm_{i,j}$ if $
\gm_{i,j} = \gm_{i,s} \circ \gm_{s,j}$.

The \textbf{length} $\len{\pth}$ of a path $\pth$ is the number of
its edges. 
Formally, we consider also paths of length $0$, which we call
\textbf{empty paths}. The \textbf{weight} $w(\pth)$ of a path
$\pth$ is defined to be the tropical product of the weights of all
the edges $(i_k, j_k)$ composing $\pth$, counting repeated edges. The weight of an empty path is formally set to be $0$.

A path is \textbf{simple} if each vertex appears at most once.
(Accordingly,  an empty path is considered also as simple.) A path
that starts and ends at the same vertex is called a
\textbf{cycle}; an edge $\rho = (i,i)$ is called a
\textbf{self-loop}, or \textbf{loop} for short. We write
$(\rho)^k$ for the composition $\rho \circ \dots \circ \rho$ of a loop
$\rho$ repeated $k$ times, and call it a \textbf{multiloop}. The
notation $(\rho)^0$ is formal, and stands for an empty loop, which
can be realized as a vertex.

\begin{remark}\label{rmk:acyclic} When a matrix $A$ is triangular, its associated digraph $G_A$ is an
acyclic digraph, possibly with  loops. Since this paper concerns
only with triangular matrices, in what follows  we assume
\textbf{all graphs are acyclic digraphs}.

Given a path $\gm_{i,j}$ from $i$ to $j$ in an acyclic digraph
$G_A$, it contains a unique simple path from $i$ to~$j$, denoted
$\tgm_{i,j}$, where the remaining edges are all loops. Relabeling
the vertices of $G_A$, we may always assume that $i<j$ and thus
have $\ell(\tgm_{i,j}) \leq j-i$.
\end{remark}

It is well known that powers of a tropical matrix $A = (a_{i,j})$
correspond to paths of maximal weight in the associated digraph,
i.e., the $(i,j)$-entry of $A^m$ corresponds to the highest weight
of all the paths~$\gm_{i,j}$ from $i$ to $j$  of length $m$ in
$G_A$.

When dealing  with product $A_1 \cdots A_m$ of different $n \times n$ matrices
the situation becomes more complicated. Namely, we have to equip
the weighted edges $e_i \in \arc_i$ of each digraph $G_{A_i}$ with
a unique color, say $c_i$, and define the digraph
$$\grph_{A_1 \cdots A_m} := \bigcup G_{A_i},$$ whose vertex set is
$\{1, \dots, n \}$ and its edge set is the union of edge sets
$\arc_i$ of $G_{A_i} = (\ver, \arc_i)$ colored by the $c_i$'s,
where $i =1,\dots,m$. (Thus, $\grph_{A_1 \cdots A_m} $ could have
multiple edges, but with different colors.)  We called such a
weighted digraph a \textbf{colored digraph}.

 Then, having such coloring, the
$(i,j)$-entry of the matrix product $B = A_1 \cdots A_m$
corresponds to the highest weight of all colored paths $(i_1,
j_1), \dots,  (i_m, j_m) $ of length $m$ from $i = i_1$ to $j=
j_m$ in the digraph $\grph_{A_1 \cdots A_m}$, where each edge
$(i_k, j_k)$ has color $c_k$, $k =1, \dots,m$, i.e., every edge is
contributed uniquely by the associated digraph $G_{A_k}$ of $A_k$, respecting the color ordering.
We call such a path a \textbf{proper colored path}.

 In what follows when considering paths
in colored digraphs $\grph_{A_1 \cdots A_m}$, we always restrict
to those colored paths that respect the sequence of coloring $c_1,
\dots, c_m$, i.e., properly colored, determined by the product concatenation $A_1 \cdots
A_m$. (For this reason we often preserve the awkward notation
$\grph_{A_1 \cdots A_m}$ which records the product concatenation $A_1 \cdots
A_m$.)

\begin{notation}\label{nott:1}
Given a matrix product $B = A_1 \cdots A_m$, we  write $\setB$ to
indicate that $B$ is realized as a word ``restoring'' the product
concatenation $A_1 \cdots A_m$, and thus denote $\grph_{A_1 \cdots
A_m}$ as $\grph_\setB$, while $B$ denotes the result of the matrix product. We also write  $\setB = \setBp \setBpp$
for the product concatenation of  $\setBp = A_1 \cdots A_k$ and
$\setBpp = A_{k+1} \cdots A_m$, with $ 1 < k < m$.
\end{notation}

%
%




\section{Semigroup identities}\label{sec:3}

\subsection{Semigroup elements}
Assuming that  $S := (S, \cdot \ )$ is a multiplicative monoid  with identity element
$e_S$, we write $s^i$ for the  $s \cdot s \cdots s$ with $s$
repeated $i$ times and formally identify $s^0$ with $e_S$.

Let $\varX$ be a countably infinite set of ``variables" (or letters) $x_1, x_2,
x_3, \dots$, i.e., $\varX: = \{ x_i \ds : i \in \N \}$.  An
element $w$ of the free semigroup $\varX^{+}$ generated by $\varX$
is called a \textbf{word} (over $\varX$), written uniquely as
\begin{equation}\label{eq:word} w =  x_{i_1} ^{t_1} \cdots
x_{i_m}^{t_m} \in \varX^{+}, \qquad i_k \in \N, \ t_k \in
\N,
\end{equation}
where $ x_{i_k} \neq  x_{i_{k+1}}$ for every $k$.   We write
$\no_{x_i}(w)$ for the number of occurrences of the variable  $x_i \in \varX$ in the word
$w$. Then
$$\cont(w) := \{ x_i \in \varX \ds | \no_{x_i}(w) \geq 1\} $$ is called the
\textbf{content} of $w$ and $$\wlen(w) := \sum_{x_i \in \cont(w)}
\no_{x_i}(w)$$ is the \textbf{length} of $w$. A word $w \in
\varX^+$ is said to be \textbf{finite} if $\wlen(w)$ is finite. We
assume that the empty word, denoted  $e$, belongs to $\varX^+$ and
set $\wlen(e) = 0$. A word $w$ is called
\textbf{$k$-uniform} if each letter ${x_i \in \cont(w)}$ appears in $w$ exactly $k$ times, i.e., $\no_{x_i}(w) = k$ for all ${x_i \in \cont(w)}$. We say that $w$ is \textbf{uniform} if it is $k$-uniform for some $k$.

We say that $w_2 \in \varX^+$ is a \textbf{factor} of a word $w\in
\varX^+$, written $\factor {w_2}{ w}$, if $w = w_1 w_2 w_3$ for
some $w_1, w_3 \in \varX^+$. When $w =
w_1 w_2$, 
we call the factors $w_1$ and $w_2$ respectively the
\textbf{prefix} and \textbf{suffix} of ~$w$, denoted as $\pre(w)$
and $\suf(w)$. Given a word  $w \in \varX^+$ we write
$\pre_{x_i}(w)$ (resp. $\suf_{x_i}(w)$) for the prefix (resp.
suffix) of $w$ of maximal length that consists only the variable
$x_i$, in particular, when $w \neq e$, $\pre_{x_i}(w) = x_i^{j_i}$ for some $x_i$ and $j_i \in
\Z_+$.

A word $u$  is a \textbf{subword} of $v$, written $u \subseteq v$,
if $v$ can be written as $v = w_0 u_{1} w_1 u_{2} w_2 \cdots u_{m}
w_m$ where $u_i$ and $w_i$ are words (possibly empty) such that $
u = u_1u_2 \cdots u_m $, i.e., the $u_i$ are factors of $u$. Clearly, any factor of $v$ is also a subword, but not conversely.

Given a finite subset $P \subset \N$, we define the ``down closure'' of $P$ to be
$$\overline{P} := \{ t \in \N \ds | t \leq p \text{ for some } p \in P\}. $$
The \textbf{exponent set} $\exp(w)$  of a word $ w$ of the Form
\eqref{eq:word} is defined as
$$ \exp(w) := \{ t_k \ds
| t_k > 0 \}.$$
In this paper we always assume all words  are  finite;
 thus $|\exp(w)|$ is finite for any word $w$. When $\overline{\exp(w)} = \exp (w) = \{ 1, \dots, m
\} $ we say that $w$ is a word \textbf{of exponent~$\setm$}.

\pSkip
\emph{Henceforth, we always assume that $P = \overline{P} \subset \N$ is a nonempty subset of the form}  $$ P := \{1,\dots,m \} \dss{ and \ that } n \geq   m = \Pmax \ .$$
\pSkip

Given finite nonempty subsets $C \subseteq \varX$ and  $P  \subset \N$, for any $n \geq  m $, $n \in \N$,  we define
$$ \wset_n [C,P] := \{ w \in \varX^+ \ds | \cont(w) \subseteq C , \ \exp(w) \subseteq P , \ \wlen(w) = n \},$$
in particular $\wset_n [C,P]  \subset C^+.$

We denote by $\tlw_{(C,P,n)}$  a word in $C^+$ for which
every  member of $\wset_n [C,P]$ is a factor, i.e.,
\begin{equation}\label{eq:uniword}
\text{ $\tlw_{(C,P,n)} \in C^+$ such that  $\factor {u} {
\tlw_{(C,P,n)}}$ for every  $u \in \wset_n [C,P]$.}
\end{equation}
We call $\tlw_{(C,P,n)}$ an $n$-\textbf{power word} of $C$ and
$P$. We say that $\tlw_{(C,P,n)}$  is \textbf{faithful} if
 $$\cont(\tlw_{(C,P,n)}) =C \dss{\text{and}}
\exp(\tlw_{(C,P,n)}) = P \ .$$  Note that $\wlen(\tlw_{(C,P,n)}) \geq n$, while for  $|C| > 1$ we have $\wlen(\tlw_{(C,P,n)}) > n$.  When $|C| = 1$, say $C = \{ x_i\}$, then
$ x_i^t$ is an  $n$-power word for any $t \geq n$, and $ x_i^t$ is faithful only if  $t = n = \Pmax$. \pSkip

\begin{example}\label{exmp:2-word}
Suppose $C = \{ x, y \} $, $P = \{ 1, 2 \} $.
\begin{enumerate} \eroman
  \item  When  $n =2$ we have the set $$
\wset_2 [C,P] = \{ x^2, \ xy, \ yx, \ y^2\}
\subset C^+,
$$ for which
$$ \tlw_{(C,P,2)} = x^2y^2x$$
 is a faithful $2$-power (nonuniform) word of  $C$ and $P$  of  length $5$.  \pSkip

  \item If  $n=3$ we get the set

 $$
\wset_3 [C,P] = \{ x^2y, \ xyx, \ xy^2, \ yx^2, \ yxy, \ y^2x\}
\subset C^+,
$$ for which
$$ \tlw_{(C,P,3)} = \l3 $$
 is a faithful $3$-power uniform word of  $C$ and $P$ of  length $8$.
\end{enumerate}

\end{example}

\begin{remark}\label{rmk:extenPW} Given an $n$-power word $\tlw_{(C,P,n)}$ of $C$ and
$P$, it is easy to see that for any $w_1, w_2 \in
C^+$, the word of the form $\tlw'_{(C,P,n)} = w_1
\tlw_{(C,P,n)} w_2$ is also an $n$-power word of $C$ and $P$. Therefore, taking appropriate  $w_1$ and $w_2$, $\tlw_{(C,P,n)}$ can be extended to a uniform $n$-power word. Similar extension can be performed for faithful $n$-power words, preserving their faithfulness.
\end{remark}

\pSkip
\emph{In what follows, we always work with $n$-power words $\tlw_{(C,P,n)}$ which are faithful and with  $|C| > 1.$
}
\pSkip

An $n$-power word $\tlw_{(C,P,n)}$ of $C$ and $P$ is called a
\textbf{minimal $n$-power word} if $\ell(\tlw_{(C,P,n)}) \leq
\ell(\tlw'_{(C,P,n)})$ for any $n$-power word $\tlw'_{(C,P,n)} \in
C^+$.
\begin{example}
The power words in Example \ref{exmp:2-word} are minimal power words.
\end{example}
\begin{remark}\label{rmk:simplePathWords}

Given a word $w \in \varX^+$, we may consider  $\varX$ to be a set
of generic matrices $A_1, A_2, \dots $. Using
Notation~\ref{nott:1}, the word $w = A_{i_1} \cdots A_{i_m}$ can
be realized as a  product concatenation of matrices and thus can  be
written equivalently as $\setB = A_{i_1} \cdots A_{i_m}$, where the result of the matrix product is  $B = (b_{i,j})$. Then every entry
$b_{i,j}$ of $B$ corresponds to a proper colored path $\gm_{i,j}$ from
$i$ to $j$ of length $m$ in the digraph $\grph_\setB,$ cf.~
\S\ref{sec:mathGr}. Thus, the proper coloring of all paths of
length $m$ is uniquely determined by $w$. Conversely, given a properly
colored path $\gm_{i,j}$ of length $\ell(w)$ in $\grph_\setB$, one
can recover the word $w$ from the coloring of the edges consisting
$\gm_{i,j}$.
\end{remark}

\subsection{Semigroup identities}\label{sec:3.2}
A (nontrivial) \textbf{semigroup identity} is a formal equality
 of the form  \begin{equation}\label{eq:Id}
\Id : u=v,
\end{equation} where
$u$ and $v$ are two  different (finite) words of the Form
\eqref{eq:word} in the free semigroup $\varX^{+}$. For a monoid
identity, we allow $u$ and $v$ to be the empty word as well. We
discuss, for simplicity, semigroup identities, but minor changes
apply to monoid identities as well.

 A semigroup $S := (S, \cdot \;)$ satisfies the
semigroup identity~\eqref{eq:Id} if for every morphism $\phi
:\varX^{+} \rightarrow S$ one has $ \phi(u)= \phi(v)$.

\begin{remark}
Semigroup identities can be thought of as special case of
polynomial identities (PI's), namely as monomial identities.
\end{remark}

 We say that an identity $\Id: u =v $ is an
\textbf{$n$-variable identity} if $|\cont(u) \cup \cont(v)| = n$.
The \textbf{exponent set}, denoted  $\exp(\Id)$,  of $\Id$ is
defined to be $\exp(u) \cup \exp (v).$ An identity $\Id$ is  said to be
\textbf{balanced} if $ \no_{x_i}(u) = \no_{x_i}(v)$ for every $x_i
\in \varX$, and it is called \textbf{uniformly balanced} if furthermore~$u$ and~$v$ are $k$-uniform for some $k$.    We define the \textbf{length} $\ell(\Id)$  of $\Id$ to
be $\ell(\Id) :=  \max\{\ell(u), \ell(v)\}$. It is readily checked
that if $\Id$ is balanced, then $\ell(u) = \ell (v)$.

\begin{example} Let us give some very elementary examples of semigroup identities.
\begin{enumerate} \eroman
    \item A commutative semigroup  admits the $2$-variable uniformly balanced identity $\Id: xy =
    yx,$ whose exponent set  is ~$\exp(\Id) = \{1 \}$. \pSkip

    \item An idempotent semigroup admits the (non-balanced) $1$-variable identity $\Id : x^2 =
    x,$ whose content is $\cont(\Id) = \{ x \}$ and its exponent set  is $\exp(\Id) = \{1 , 2 \}$. \pSkip
    \item A virtually abelian  semigroup  admits the $2$-variable uniformly  balanced identity $\Id: x^n y^n =
    y^n x^n,$ whose exponent set  is $\exp (\Id) = \{ n \}$. \pSkip

     \item The $2$-variable identity $\Id: x^{i} y^j =
    y^j x^{i},$   whose exponent set  is $\exp (\Id) = \{ i,j \}$, is balanced but not uniformly balanced for nonzero $i \neq j$.
\end{enumerate}
\end{example}

An identity $\Pi : u =v$ is called a \textbf{minimal identity} of
the semigroup $S$ if $\ell(\Id) \leq \ell(\Id')$ for any
nontrivial identity $\Pi' : u' =v'$ of $S$.

\begin{remark} Let $
\tI$ be a set of semigroup identities. The set of all semigroups
satisfying every identity in $\tI$ is denoted by $V[\tI]$ and is
called the \textbf{variety of semigroups} defined by $\tI$. It is
easy to see that $V[\tI]$ is closed under subsemigroups,
homomorphic images, and direct products of its members. The famous
Theorem of Birkhoff says that conversely, any class of semigroups
closed under these three operations is of the form $V[\tI]$ for
some set of identities $\tI$.
\end{remark}

\subsection{Construction of semigroup identities}
Although the main part of this paper utilizes $2$-variable identities of exponent
$\set2$, for future study,  we present the construction of identities that are of our interest in full generality.

Given an $n$-power word  $\tlw_{(C,P,n)}$ as in
\eqref{eq:uniword}, with  $C = \{ x_{1}, \dots, x_{m}  \}$ and $P
= \{ t_1, \dots, t_j\}$, we aim to build a nontrivial balanced identity $\Pi : u =v$ of content $C$ and exponent set $P$. To preserve the exponent set $P$ for $\Pi$, if necessary, we first extend $\tlw_{(C,P,n)}$ to  $\tlw'_{(C,P,n)}$ (cf.  Remark \ref{rmk:extenPW}), and construct the words $u$ and $v$ such that $\tlw_{(C,P,n)}'$ is their prefix and suffix.

Let
$t_{\max} := \max \{ t_1, \dots, t_j\}$, $t_{\min} := \min \{ t_1,
\dots, t_j\}$, and let $d := t_{\max} - t_{\min}$.  We assume the
following: $$    |C| > 1, \quad
    |P| >1, \quad
    t_{\max} \geq 2 t_{\min} \ .$$

 Letting
$$z_1 :=  x_1^{t_{\min}} \cdots  x_m^{t_{\min}},  \qquad  z_2 := x_m^{t_{\min}} \cdots  x_1^{t_{\min}} \ , $$
we define the identity
\begin{equation}\label{eq:iduniv}
 \Id_{(C,P,n)}: \quad   \tlw'_{(C,P,n)} \ds {z_1} \tlw'_{(C,P,n)} \ds =  \tlw'_{(C,P,n)} \ds
 {z_2}
\tlw'_{(C,P,n)} ,
\end{equation}
 where $ \tlw'_{(C,P,n)} $ is defined as
\begin{equation}\label{eq:idunivRole}  \tlw'_{(C,P,n)}  : = w_1 \ds{ \tlw_{(C,P,n)} } w_2,
\end{equation}
with  $w_1$ and $w_2$ given as follows (letting  $\tlw := \tlw_{(C,P,n)}$, for short):
\begin{equation}\label{eq:idunivRole1} w_1 : =
\left \{  \begin{array}{ll}
        x_m &  \text{if }   \pre_{x_{1}}(\tlw) > d  \\[1mm]
        x_1 &  \text{if }   \pre_{x_{m}}(\tlw) > d  \\[1mm]
        e &  \text{otherwise} \end{array}\right. \qquad w_2 :=
 \left \{ \begin{array}{ll}
        x_m &  \text{if }   \suf_{x_{1}}(\tlw) >  d  \\[1mm]
        x_1 &  \text{if }   \suf_{x_{m}}(\tlw) > d  \\[1mm]
        e &  \text{otherwise} \\
          \end{array} \right.
\end{equation}
Clearly, by this construction,  $\Id_{(C,P,n)}$ is a
balanced identity.

In view of Remark~\ref{rmk:extenPW}, $\tlw'_{(C,P,n)}$ can be extended further to be uniform, which then makes  the identity~\eqref{eq:iduniv} uniformly balanced.
In particular when the given $n$-power word  $\tlw_{(C,P,n)}$ is uniform, we can instead explicitly define $w_1$ and $w_2$ in \eqref{eq:idunivRole} as
\begin{equation}\label{eq:idunivRole2} w_1 :=  \left \{  \begin{array}{ll}
        z_1 &  \text{if }   \pre_{x_{1}}(\tlw) > d  \\
        z_2 &  \text{if }   \pre_{x_{m}}(\tlw) > d  \\
        e &  \text{otherwise} \end{array}\right. \qquad w_2 : = \left \{ \begin{array}{ll}
        z_2 &  \text{if }   \suf_{x_{1}}(\tlw) >  d  \\
        z_1 &  \text{if }   \suf_{x_{m}}(\tlw) > d  \\
        e &  \text{otherwise} \\
          \end{array} \right.
\end{equation}
to obtain a uniformly balanced identity.

\begin{notation}\label{nott:2} Throughout this paper we use the notation $\x$ and $\y$ to mark specific instances
of the variables $x$ and $y$ in a given expression, although these notations stand for the same variables $x$ and $y$, respectively.
\end{notation}

\begin{example}\label{exmp:2-word:2}
Let  $C = \{x,y\}$ and $P = \{ 1,2\}$, and set $z_1 = \underline{x y}$, $z_2 = \underline{y x}$.
\begin{enumerate} \eroman
  \item
Starting with the $2$-power word $ \tlw_{(C,P,2)} = x^2y^2x$
of $C$ and $P$ given in
Example~\ref{exmp:2-word}.(i), by the rule of \eqref{eq:idunivRole1} we extend it to  $ \tlw'_{(C,P,2)} = \ll2$, a uniform word, and define the identity
\begin{equation}\label{eq:id2} \Id_{(C,P,2)}: \quad
    \ll2 \ds{\underline{xy}}  \ll2 \ds =  \ll2 \ds{\underline{yx}}
     \ll2 \ .
\end{equation}
This identity is  uniformly balanced.
\pSkip

  \item
Taking the uniform $3$-power word $ \tlw_{(C,P,3)} = \l3$
of $C$ and $P$ as in
Example~\ref{exmp:2-word}.(ii), we get the uniformly balanced identity
\begin{equation}\label{eq:id3} \Id_{(C,P,3)}: \quad
     \l3 \ds{\underline{xy}}  \l3 \ds =  \l3 \ds{\underline{yx}}
     \l3 \ .
\end{equation}
(In this case, by \eqref{eq:idunivRole1} there is no need for extension.)
\end{enumerate}

For both identities, $\Id_{(C,P,2)}$ and  $\Id_{(C,P,3)}$, we have
$\cont(\Id_{(C,P,n)}) = \cont(\tlw'_{(C,P,n)}) = \cont(\tlw_{(C,P,n)})$ and
$\exp(\Id_{(C,P,n)})  = \exp(\tlw'_{(C,P,n)}) = \exp(\tlw_{(C,P,n)}),$ for $n =2,3.$
\end{example}
\begin{theorem}\label{thm:refine} A semigroup  $S:= (S, \cdot \;)$ that satisfies an
 $n$-variable identity $\Id: u = v$, for $n \geq 2$,  also satisfies  a refined
$2$-variable identity $\htId : \htu =\htv$ of exponent $\set2.$
\end{theorem}

\begin{proof} Since  $S$ satisfies the $n$-variable identity
$\Id : u =v $, then by definition $\phi(u) = \phi(v)$ for  every
morphism $\phi :\varX^{+} \rightarrow S$. Suppose $ C: =
\cont(\Id) = \{x_1, \dots, x_n\}$, and write $C$ as the disjoint
union  $C = C_1 \ds { \dot \cup}  C_2$, for nonempty subsets $C_1$
and $C_2$. Pick two variables, say $y_1, y_2$,  and
consider the words $\htu$ and $\htv$, obtained respectively from
$u$ and $v$ by substituting $y_1 y_2$ for every $x_i \in C_1$ and
$y_2 y_1$ for every $x_j \in C_2.$

It is easy to verify that $ \htId : \htu = \htv$ is a $2$-variable
identity, with set of exponents $\exp(\htId)  \subseteq \{ 1,2\}.$ We claim
that $S$ satisfies the identity $\htId :  \htu = \htv$. Indeed,
assume $\phi : \varX^+ \to S$ sends $\phi : y_1 \mapsto s_1$ and $
\phi :y_2 \mapsto s_2$, then $\phi : y_1 y_2 \mapsto  s_1 s_2 =a $
and $\phi : y_2 y_1 \mapsto  s_2 s_1 =b$. But, since $a$ and $b$
satisfy $\Id$ by hypothesis, and $\htu$ and $\htv$ can be decomposed
as concatenation of the terms  $y_1 y_2$ and $y_2 y_1$, then $s_1$
and $s_2$ satisfy $\htId$.
\end{proof}

In the sequel, in view of Theorem \ref{thm:refine}, we focus on
$2$-variable identities $ \Id_{(C,P,n)}$ of exponent~
$\set2$ of the Form \eqref{eq:iduniv}, with $C := \{ x,y\}$ and $P
:= \{ 1,2\}.$  For ease of exposition, for a given $n$-power word $\tlw_{(C,P,n)}$ we begin with the 2-variables identity having
exponent set $\set2$ of the form
\begin{equation}\label{eq:iduniv2}
\Id_{(C,P,n)}: \quad   \tlw'_{(C,P,n)} \ds {\x}
\tlw'_{(C,P,n)} \ds =  \tlw'_{(C,P,n)} \ds
 {\y}
\tlw'_{(C,P,n)} ,
\end{equation}
where here, using Remark \ref{rmk:extenPW}, $\tlw' :=
\tlw'_{(C,P,n)}$ is an extended $n$-power word, obtained by the
rule of Formula~ \eqref{eq:idunivRole1}; in particular,
$\pre_{x}(\tlw')$, $\pre_{y}(\tlw')$,  $\suf_{x}(\tlw')$,
$\suf_{y}(\tlw')$ are all $\leq 1$, which preserve the exponent~$\set2$ property of $\Id_{(C,P,n)}$. (Note that in comparison  to \eqref{eq:iduniv}, the intermediate terms $z_i$ are now consisting of only one letter.)

Clearly the identity $\eqref{eq:iduniv2}$ is
not balanced, however it  can be easily refined  by substituting
\begin{equation}\label{eq:subtitiue} \text {$x:= \chx \ds{\chy} $ \dss{and}  $y :=
\chy \ds{\chx}$}
\end{equation} to receive back balanced identity (of exponent $\set2$)
\begin{equation}\label{eq:iduniv3}
\widecheck{\Id}_{(\chC,P,n)}: \quad   {\chw'}_{(\chC,P,n)} \ds {\chx
\chy} \chw'_{(\chC,P,n)} \ds =  \chw'_{(\chC,P,n)} \ds
 {\chy \chx}
\chw'_{(\chC,P,n)},
\end{equation}
  with $\chC = \{ \chx,\chy\}$, $P = \{ 1,2 \}$, and  $\chw'_{(\chC,P,n)}$ is the word obtained from $\tlw'_{(C,P,n)}$ substitution \eqref{eq:iduniv2}. (Note that now $\tlw'_{(C,P,n)}$ need not be an $n$-power word of $\chC = \{ \chx,\chy\}$ and $P = \{ 1,2 \}$.)


\section{Identities of triangular tropical matrices}
Aiming to prove the existence  of a semigroup identity for the
monoid $ U_n(\Trop)$ of $n \times n$ triangular matrices, we start with the case of diagonally
equivalent matrices, which is easier to deal with; then we
generalize the results to the whole monoid  $ U_n(\Trop)$.

\begin{remark}
 Suppose $S = M_n(\Trop)$ is the monoid of all $n \times n$
tropical matrices, then  any semigroup identity $\Id : u =v$ that
$S$ admits is balanced. Indeed,  otherwise assume $\no_{x_i}(u)
\neq
    \no_{x_i}(v)$ for some $x_i$ and take morphism $\phi: x_j \mapsto I$
    for each $j \neq i $ (recall that $I$ is the identity matrix) and $\phi : x_i \mapsto \al I$ for some fixed $\al
    \neq
    \one$ to reach a contradiction.

    However, for an easy exploration, for the certain class of diagonally equivalent matrices,
    we first work with unbalanced identities for the Form  \eqref{eq:iduniv2} and then refine
    them to balanced identities as in~
    \eqref{eq:subtitiue}.
\end{remark}

When dealing with matrix identities, we sometimes denote generic matrices
(standing for variables $x, y, \dots$) by capital letters $X, Y,
\dots$, as well as the words they generate. To demonstrate the main idea
of our approach for proving existence of semigroups identities for tropical matrices, we first prove the existence of a semigroups identity admitted by the monoid of $2 \times 2 $ triangular tropical matrices.

\subsection{The monoid of $2 \times 2$ triangular tropical  matrices}
An explicit semigroup identity of the case of $2\times 2$ tropical matrices has been proven in
\cite{IzhakianMargolisIdentity} by using the machinery of Newton's
polytope, applied to generic matrices, which allows the
identification of different tropical polynomials describing the
same function. Essentially, this machinery transfers the
identification  problem to the realm of convex sets, reducing it
to a comparison of convex hulls. However, computing the
convex-hulls becomes  difficult in high dimensional cases, and it
is not easily applicable.

To demonstrate our new approach,  we bring a simpler proof of the
above case, based now on the colored paths in associated digraphs. This proof is given for an
explicit identity of the form \eqref{eq:iduniv2}.

\begin{theorem}\label{prop:2x2}
Any two matrices  $X,Y \in U_2(\Trop)$ such that  $X \dsim Y$
satisfy the identity
\begin{equation}\label{eq:explicit2X2} U := X Y \ \X \  X Y \  = \
X Y \ \Y \  X Y =: V
\end{equation}
of the Form  \eqref{eq:iduniv2}.
\end{theorem}
\begin{proof}
Write $U = (u_{i,j})$ and $V = (v_{i,j})$.    The equality $u_{i,i}
= v_{i,i}$ is obvious for the diagonal entries.
 Consider the $(1,2)$-entry, say of the matrix product  $U$, this entry
 corresponds to a colored path $\gm_{1,2}$ from $1$ to $2$ of
maximal weight  in the digraph $G_{X Y \X X Y} \ (= G_\setU )$.
Clearly, $\gm_{1,2}$  is of length $5$ and it contains a simple
path $\tgm_{1,2} \subset \gm_{1,2}$ of length $1$; namely the simple path $\tgm_{1,2}$  is an
edge. Let $e_\X$ be the edge in $\gm_{1,2}$ contributed by $G_\X$ (or equivalently by $\X$).
If $e_\X$ is a loop, we are done since $e_\X$ can be replaced by
$e_\Y$ in $G_{X Y \Y X Y} \ (= G_\setV)$, since $\X \dsim \Y$,
yielding a path of the same weight. Otherwise, by the same
argument, it is enough to show that $G_{X Y \X X Y}$ has another
path of the same length and weight in which the contribution of
$G_\X$ is a loop.

Assume $e_\X$ is not a  loop, i.e., $\tgm_{1,2} = e_\X$,  then
$\gm_{1,2}$ has the form
$$\gm_{1,2} = (\rho_1)^2  \circ e_\X \circ (\rho_2)^2, $$
where $\rho_1$ and $\rho_2$ are loops. If $w(\rho_1) > w(\rho_2)$
(resp. $w(\rho_2) > w(\rho_1)$) then the path $(\rho_1)^3  \circ
e_\X \circ \rho_2$ (resp. $e_\X \circ (\rho_2)^4 $) would have a
higher weight than $\gm_{1,2}$ has -- a contradiction. (Note that
that a loop can be contributed equivalently either by $G_X$ or
$G_Y$.) Thus $w(\rho_1) = w(\rho_2)$, and hence $e_\X \circ
(\rho_2)^4$ and $(\rho_1)^3 \circ e_\X \circ \rho_2$ are paths of
the same weight as $\gm_{1,2}$ in which $G_\X$ contributes a loop.
\end{proof}

\begin{remark} Although we have proven Theorem \ref{prop:2x2} for
the explicit identity \eqref{eq:explicit2X2} of diagonally equivalent matrices, the same proof also
holds for any identity of $2 \times 2$ triangular tropical matrices given in the general form as in
\eqref{eq:iduniv2}.
\end{remark}

\begin{corollary}[{\cite[Theorem~3.6]{IzhakianMargolisIdentity}}] Any matrices $A,B \in U_2(\Trop)$
satisfy the  semigroup identity  \begin{equation}\label{eq:2X2} A
B^2 A \ A B \ A B^2 A \ds =  A B^2 A \  BA  \ A B^2 A.
\end{equation}
\end{corollary}
\begin{proof}
Apply Theorem \ref{prop:2x2} for $X := AB$ and  $Y := BA$ as in
\eqref{eq:subtitiue}.
\end{proof}
One easily sees that \eqref{eq:2X2} is a 2-variable uniformly  balanced
identity of length $10$. Moreover, by
\cite{IzhakianMargolisIdentity}, we know that this is an identity
of minimal length that $ U_2(\Trop)$ satisfies.

\subsection{The general case}

We now turn to prove the existence of a semigroup identity in the general case of $n \times n$  triangular tropical matrices, first generalizing Theorem
\ref{prop:2x2} to identities of the Form \eqref{eq:iduniv2} for $n
\times n$ diagonally equivalent triangular matrices, which later provides the proof for all $n \times n $  triangular matrices.

\begin{remark}\label{rmk:n-1} Viewing the entries of a product of $n\times n$ triangular tropical  matrices as colored paths in the associated digraph, cf. \S\ref{sec:mathGr}, one gets only paths containing simple subpaths of length $< n$. Therefore, concerning identities of the Form $\eqref{eq:iduniv2}$,  applied to $n \times n$ triangular matrices, it is enough to
implement an identity  $ \Id_{(C,P,n-1)}$ constructed by using $(n-1)$-power words.
\end{remark}

Given two $n\times n$ triangular matrices $X \dsim Y$,  let $Z \in
\uT$ be the product concatenation
\begin{equation}\label{eq:oneProd} \setZ :=  \setL  \ds    \X
\setR,   \qquad \setL =  \setR  = \tlw_{(C,P,n-1)},
\end{equation}
%
with $C := \{ X ,Y\}$ and $P := \{ 1,2\}.$ Recall that the
notation $\X$ is used to mark the specific instance of the matrix
$X$ in the expression, although it is just the same matrix as $X$,
and that  $\setZ$ stands for   the product concatenation (i.e., a formal word) whose product result is the matrix   $Z = (z_{i,j})$
  given in \eqref{eq:oneProd}.

 In the view of \S\ref{sec:mathGr}, the $(i,j)$-entry $z_{i,j}$ of $Z = (z_{i,j})$ corresponds
to a colored path $ \gm_{i,j}$ of maximal weight and length
$\ell(\gm_{i,j}) = \ell(\setZ)$ from $i$ to $j$ in the associated
digraph $\grph_\setZ$ of  $\setZ$. Then, for $j > i $,  cf. Remark~
\ref{rmk:acyclic}, the simple colored subpath $\tgm_{i,j} \subset
\gm_{i,j}$  from $i$ to $j$ is of length $ \leq  j-i \leq n-1$, and thus
$\gm_{i,j}$ contains exactly $\ell(\setZ) - \ell(\tgm_{i,j})$
loops.

\begin{remark} The matrix product concatenations $\setL$ and  $\setR$
in  \eqref{eq:oneProd}  have been taken to be $\tlw_{(C,P,n-1)}$ --
the $(n-1)$-power word of $C$ and $P$ for which every member  of
$\wset_{n-1} [C,P]$ is a factor. This allows us to deal also with
cases in which the involved matrices have diagonal entries $\zero$
which means that some vertices in the associated digraphs are not
adjunct to a loop.

In the sequel exposition, when working with matrices whose
diagonal entries are all nonzero, we may replace $\tlw_{(C,P,n-1)}$
by a word for which every member of $\wset_{n-1} [C,P]$ is a subword (and not necessarily a factor), to possibly obtain shorter semigroup identities.
\end{remark}

Given a simple path $\tgm_{i,j}$ from $i$ to $j$, we write
$W(\tgm_{i,j})$ for the word recorded uniquely by the  coloring of
$\tgm_{i,j}$. In particular $\ell(W(\tgm_{i,j})) = \ell
(\tgm_{i,j})$. \pSkip

The next lemma plays a central role in this paper.

\begin{lemma}\label{lem:main}  Suppose $\setZ $ is as in \eqref{eq:oneProd}, where $X,Y\in \uT$. Let $\gm_{i,j}$, where $i <  j$,
be a colored path  of maximal weight  in  $\grph_\setZ$ for  which
the contribution of $\grph_\X$  is
 a non-loop edge $e_\X$. Then  $\grph_\setZ$ has another colored path of the same length and weight  in which
 the contribution of  $\grph_\X$ is a loop.
\end{lemma}

The proof of the lemma is quite technical, thus, before proving it formally, let us outline its major idea. Given a path $\gm_{i,j}$ from $i$ to $j$, it contains a unique simple subpath $\tgm_{i,j}$ from $i$ to $j$ which corresponds to a subword $W$ of  $\setZ$; all the other edges of  $\gm_{i,j}$  are loops. We want to show that if the edge $e_\X$ contributed by $G_\X$ appears in $\tgm_{i,j}$, then $e_\X$ can be excluded from $\tgm_{i,j}$ by ``shifting'' $\tgm_{i,j}$ in one of three ways without changing its weight: either shifting  $\tgm_{i,j}$ to the left or to the right such that $W$ becomes a factor of $\setL$ or $\setR$ respectively, or by writing $\tgm_{i,j}$ as a composition $\tgm_{i,k} \circ \tgm_{k,j}$, and $W = W_1 W_2$ correspondingly, and shifting $\tgm_{i,k}$ to the left and $\tgm_{k,j}$ to  right such that both $W_1$ and $W_2$ are factors of  $\setL$ and  $\setR$ respectively. As a consequence of these  ``shifts'', which are possible since on each vertex $G_X$ and $G_Y$ have loops of the same weight,  the contribution of  $G_\X$ to  $\gm_{i,j}$ becomes a loop, while the edge $e_\X$ is replaced in $\tgm_{i,j}$ by $e_X$, contributed by another $G_X$.

\begin{proof}[Proof of Lemma \ref{lem:main}]

Let  $m  = \ell(\setR) = \ell(\setL)$,  since $\setL=  \setR  = \tlw_{(C,P,n-1)}$ -- an $(n-1)$-power word, then $m \geq n$, and thus $\ell(\gm_{i,j}) = 2m +1$.  Let $\tlm =
\ell(\tgm_{i,j})$. In particular,  $\tlm < n$ since
   $\tgm_{i,j}$ is a simple path in
$\grph_\setZ$ -- an acyclic digraph on $n$ vertices, and thus $\tlm < m$. Recall that by hypothesis
$e_\X \in \tgm_{i,j}$ . The proof is delivered by cases, determined by
the structure of the path~$\gm_{i,j}$.

Write \begin{equation}\label{eq:1} \gm_{i,j} = \tgm_{i,s} \circ
(\rho_s)^{p_s} \circ  \hgm_{s,t} \circ (\rho_t)^{p_t} \circ
\tgm_{t,j}, \qquad i \leq s < t \leq j,
\end{equation}
where $\tgm_{i,s}$ (resp. $\tgm_{t,j}$) is the maximal simple path
(could be empty) appearing as the prefix (resp. suffix) of $
\gm_{i,j}$,  $\rho_s$ and $\rho_t$ are loops which must exist due to length considerations, and $\hgm_{s,t}$ is a subpath (needs not be simple).

Therefore,  $e_\X$ does
not belong to $\tgm_{i,s}$ nor $\tgm_{t,j}$, by length considerations, and thus
$\ell(\tgm_{i,s}), \ell(\tgm_{t,j}) < \tlm-1 $ and hence $p_s, p_t
> 0.$

Define  $\setF := \word{ \tgm_{i,s}} $ and  $\setG := \word{
\tgm_{t,j}}$ to be the words (whose terms are generic matrices) determined
by the coloring of the simple paths $ \tgm_{i,s}$ and $ \tgm_{t,j},$ and set
$\setLp := \sres{\setL}{\setF}$ and $\setRp :=
\pres{\setR}{\setG}.$ In other words $\setLp$ and $\setRp$ are the
words obtained from $\setL$ and $\setR$  after removing respectively the
initial and the terminal words (which in this case are factors) corresponding to the simple paths $
\tgm_{i,s}$ and $ \tgm_{t,j}$.

Let \begin{equation}\label{eq:2} \mu_{s,t} =  (\rho_s)^{p_s} \circ
\hgm_{s,t} \circ (\rho_t)^{p_t}, \qquad  s < t ,
\end{equation}
be the subpath of $\gm_{i,j}$, given by its intermediate
non-simple part according to \eqref{eq:1}, and let $\tmu_{s,t}$ be
the simple subpath contained in  $\mu_{s,t}.$ Define $\setH :=
\word{\tmu_{s,t}}$ -- the matrix product concatenation, realized
as a subword, corresponding to the coloring of the path $\tmu_{s,t}$.

We claim that  $ \grph_{\setLp}$ contains a path similar to
$\tmu_{s,t}$ and $\grph_{\setRp}$ contains a path similar
$\tmu_{s,t}$. It is enough to show that $\setH$ is a subword of
$\setLp$.  Indeed, $\setF \setH \subseteq \setL $ by word
construction, where $\setF$ is the prefix of $\setL$  by
hypotheses, thus $\setH \subseteq \setLp$ is a subword of
$\setLp$. The case of $\tmu_{s,t} \subset \grph_{\setRp}$ is dual.

Let $\rho_{\max}$ denote the loop of maximal weight in
$\mu_{s,t}$, then we have the following possible cases:

\begin{description}
    \item[I. $\rho_s = \rho_{\max}$] Then, there is a path $ \mu'_{s,t} = (\rho_{\max})^{q_s} \circ  \tmu_{s,t} \circ (\rho_t)^{q_t}
    $ with  $\tmu_{s,t} \subset \grph_{\setRp}$ and $q_s > \ell(\setLp)$, such that $w(\mu'_{s,t}) = w(\mu_{s,t})$, since otherwise
    we would get a contradiction to the maximality of weight of~$\mu_{s,t}$. Thus
    $e'_\X = \rho_s = \rho_{\max}$ -- a loop. \pSkip

    \item[II. $\rho_t = \rho_{\max}$] Then, there is a path  $ \mu'_{s,t} = (\rho_{s})^{q_s} \circ  \tmu_{s,t} \circ (\rho_{\max})^{q_t}
    $, with $\tgm_{s,t} \subset \grph_{\setLp}$ and $q_t > \ell(\setRp)$, such that $w(\mu'_{s,t}) = w(\mu_{s,t})$, since otherwise
    we would get a contradiction to the maximality of weight of~$\mu_{s,t}$. Thus
    $e'_\X = \rho_t = \rho_{\max}$ -- a loop. \pSkip

    \item[III. $\rho_{\max} \in \hgm_{s,t}$] Then, there is a path $$ \mu'_{a,t} = (\rho_{s})^{q_s} \circ  \tmu_{s,k}
    \circ (\rho_{\max})^{q_k} \circ  \tmu_{k,t}  \circ (\rho_{q_t})^{q_t}
    , \qquad s < k < t,$$ with   $\tmu_{s,k} \subset \grph_{\setLp}$ and  $\tmu_{k,t} \subset
    \grph_{\setRp}$ such that  $\tmu_{s,t} = \tmu_{s,k} \circ \tmu_{k,t}
    $. Thus  $w(\mu'_{s,t}) = w(\mu_{s,t})$, since otherwise
    we would get a contradiction to the maximality of weight of~$\mu_{s,t}$.
    Hence
    $\rho_k = \rho_{\max}$ and $e'_\X = \rho_k$ -- a loop.
\end{description}
Therefore,  in all the above cases we  get that $e'_\X$ is a loop
in $\mu'_{s,t}$ -- a path in $G_{\setLp \X \setRp}$. Then,  concatenate the simple paths $\tgm_{i,s} $ and $\tgm_{t,j}$
\begin{equation*}\label{eq:3} \gm'_{i,j} = \tgm_{i,s} \circ
\mu'_{s,t} \circ \tgm_{t,j}, \qquad i \leq s < t \leq j,
\end{equation*}
to obtain another path in $G_{\setL \X \setR}$  for which the contribution of  $G_\X$ is a
loop, as desired.
\end{proof}

\begin{theorem}\label{thm:simDia}Any two diagonally equivalent matrices $X,Y \in \uT$, i.e., $X \dsim Y$, satisfy the identities $\Id_{(C,P,n-1)}$ of the Form~
\eqref{eq:iduniv2}, with $x = X$ and $y =Y$.
\end{theorem}
\begin{proof} Write  $X = (x_{i,j})$ and $Y = (y_{i,j})$.  Let $U =(u_{i,j})$ and $V =
(v_{i,j})$ be the matrix products determined respectively by the
left and the right words of the identity  \eqref{eq:iduniv2}. We
need to show that $u_{i,j} = v_{i,j}$ for every $i \leq j $. (The case of $j > i$ is trivial since $x_{i,j} = y_{i,j} = -\infty$, for any  $j > i$, and hence $u_{i,j} = v_{i,j} = -\infty$.)

 It is easy to see that $u_{i,i} = v_{i,i}$ for every $i =1,\dots, n$.
Assume now that  $i< j$, and consider the associated colored digraphs
$\grph_\setU$ and $\grph_\setV$ with matrix products $U$ and~$V$,
realized as words $\setU$ and~ $\setV$  as given by
\eqref{eq:iduniv2}.  Assume first that $u_{i,j} \neq -\infty$, then the value of the entry $u_{i,j}$ corresponds
to a colored path $\gm_{i,j}$ from $i$ to $j$ of maximal weight
and of length $\ell(\setU) $ in the digraph $\grph_\setU$ . By
Lemma~\ref{lem:main} we may assume that the contribution of
$\grph_\X$ to $\gm_{i,j}$ is a loop $\rho_\X$, but then
$\grph_\setV$ also contains  a similar colored path $\gm'_{i,j}$ in
which the contribution of $\grph_\Y$
 is also a loop $\rho_\Y$,  replacing  $\rho_\X$, which by hypothesis have the same
weight, i.e., $w(\rho_\X) = w(\rho_\Y)$.  Dually, the same argument also holds for a path in $\grph_\setV$ in which the contribution of $G_\Y$ is a loop.

Suppose now that $u_{i,j} = -\infty$, and assume that $v_{i,j} \neq -\infty.$  This means that there exists a colored path $\gm'_{i,j}$ from $i$ to $j$ in $\grph_\setV$, and in particular, by Lemma \ref{lem:main}, a path in which the contribution of $\grph_\Y$ is a loop $\rho_\Y$.  But then, by the above dual argument, $\grph_\setU$ also has a similar path, and thus $u_{i,j} \neq -\infty$ -- a contradiction.

Putting all together, we have $u_{i,j} = v_{i,j}$
for every $i,j$.
%
\end{proof}

\begin{example} Assume   $X \dsim Y$, and set  $x = X $ and $y =  Y$.

\begin{enumerate} \eroman
  \item  If $ X, Y \in U_3(\Trop)$ then they  satisfy the identity
\begin{equation}\label{eq:id3x3} \Id_{(C,P,2)}: \quad
     \ll2 \ds{\underline{x}}  \ll2\ds =  \ll2 \ds{\underline{y}}
     \ll2
%
     \ .
\end{equation}
\pSkip
  \item
  f $ X, Y \in U_4(\Trop)$ then they  satisfy the identity
\begin{equation}\label{eq:id4x4} \Id_{(C,P,3)}: \quad
     \l3 \ds{\underline{x}}  \l3 \ds =   \l3 \ds{\underline{y}}
      \l3 \ .
\end{equation} \end{enumerate}
(To preserve the exponent $\set2$ of the identities, we use the power words of Example \ref{exmp:2-word} with an additional instance, denoted as $\overline{y}$, of $y$ given by the rule of \eqref{eq:idunivRole1}.)
\end{example}

\begin{theorem}\label{thm:matU}The submonoid $\uT \subset \mT$ of
upper triangular  tropical   matrices satisfies the identities $\widecheck{\Id}_{(C,P,n-1)}$ of
the Form \eqref{eq:iduniv3}, which we recall is \eqref{eq:iduniv2}
with $x = X = AB$ and $y = Y =  BA$,  for $n \times n $ generic matrices $A, B$.\end{theorem}
\begin{proof} It easy to verify that for any triangular matrices $A,
B\in \uT$ the matrix products  $X = AB$ and $Y = BA$ are
diagonally equivalent. The proof is then  completed by Theorem
\ref{thm:simDia}.
\end{proof}

\begin{example} Set  $x = AB $ and $y =  BA$.

\begin{enumerate} \eroman
  \item  The monoid $U_3(\Trop)$ of $3
\times 3$ triangular tropical matrices satisfies the identity \eqref{eq:id3x3}.
\pSkip
  \item
  The monoid $U_4(\Trop)$ of $4
\times 4$ triangular tropical matrices admits the identity \eqref{eq:id4x4}.
\end{enumerate}
 \end{example}

%
%
%
%
%
%
%

\section{Identity  length: An upper bound}\label{sec:5}
In the previous section we proved the existence of  semigroup
identities of the Form \eqref{eq:iduniv3},  satisfied by~ $\uT$,
we now discuss the length of this identity, providing a very naive
upper bound.
%
%

The very well known Fibonacci sequence $\fib_n$ is defined by the
recursive  relation
\begin{equation}\label{eq:Fib} \fib_n := \fib_{n-1} + \fib_{n-2},
\qquad \text {for every  } n \geq 2 ,
\end{equation} where
$\fib_0 = 0  $ and $\fib_1 = 1$, and has the closed formula
(known as Binet's Fibonacci number formula):
\begin{equation}\label{eq:FinClosed} \fib_n = \frac{(1 +
\sqrt{5})^n - (1 - \sqrt{5})^n} {2^n \sqrt{5}} \ .
\end{equation}
Therefore, we see that the Fibonacci number $2 \fib_n$ gives the number of
elements in $\wset_n [C,P]$, i.e., all possible factors of length
$n$, for the case of $C := \{ x,y\}$ and $P := \{ 1,2\}$. (The multiplier $2$ stands for the two possibilities for starting a sequence, either with  $x$ or $y$.)

A naive construction of an $n$-power word $\tlw_{(C,P,n)}$, i.e.,
by concatenating all factors in $\wset_{n} [C,P]$, gives us the upper bound
$$\ell(\tlw_{(C,P,n)}) \leq 2 (n  +1) \fib_{n}.$$ (The multiplier of $n+1$,
stands for the length of a factor with a possible additional letter
between sequential factors, aiming to preserve the exponent $\set2$
property.) Thus, considering the refinement given by $x = \chx \chy$ and
$y = \chy \chx$,  we have
\begin{equation}\label{eq:upper} \ell(\Id_{(C,P,n)}) \leq
8(n+1)\fib_{n} + 2.
\end{equation}

Obtusely, this rough upper bound assumes that the factors $\wset_{n} [C,P]$
 do not overlap in $\tlw_{(C,P,n)}$. Dealing with possible overlap
 factors, one can reduce the multiplier $n+1$ in
 \eqref{eq:upper}, and hence can  improve further this upper bound.

Recall that, in view of  Remark~\ref{rmk:n-1}, for $n\times n$ triangular matrices it is enough to consider identities~ $\Id_{(C,P,n-1)}$, for which the upper bound is then smaller, that is  $$\ell(\Id_{(C,P,n-1)}) \leq
8n\fib_{n-1} + 2.$$

%
%

\section{Remarks and open problems}

A natural question arisen from our identity construction in
\S\ref{sec:3.2} is about the minimality of semigroup identities
admitted by $\uT$.

\begin{conjecture} When $\tlw'_{(C,P,n-1)}$ is a minimal $n$-power word
 of  $P$ and $C$, then the identity
$\widecheck{\Id}_{(C,P,n-1)}$ in~\eqref{eq:iduniv3} is a minimal
semigroup identity admitted by  $\uT.$
\end{conjecture}

The results of this paper, and those of
\cite{IzhakianMargolisIdentity}, lead us to the conjecture, which
has already been conjectured earlier in
\cite{IzhakianMargolisIdentity},  that
\begin{conjecture} Also
the monoid $M_{n}(\Trop)$ of $n \times n $ tropical matrices satisfies a nontrivial semigroup identity for all~$n$.
\end{conjecture} \noindent
(The conjecture has been proven in \cite[Theorem 3.9]{IzhakianMargolisIdentity} for the case of $n =2.$)

Another reason for
conjecturing this is that every finite subsemigroup of
$M_n(\Trop)$ has polynomial growth \cite{Gaub, Simon}. In
particular, the free semigroup on 2 generators is not isomorphic
to a subsemigroup of $M_n(\Trop)$. While Shneerson \cite{Shn} has
given examples of polynomial growth semigroups that do not satisfy
any nontrivial identity (no such example exists for groups by
Gromov's Theorem \cite{Grom}), we conjecture that this is not the
case for $M_{n}(\Trop)$.



\bibliographystyle{abbrv}

\end{document}